%% file: main-7-3.tex
\documentclass[12 pt]{amsart}
  \usepackage[utf8]{inputenc}
  \usepackage[margin=3cm]{geometry}
\makeatletter
\def\l@subsection{\@tocline{1}{0pt}{2pc}{1pc}{}}
\def\l@subsubsection{\@tocline{2}{0pt}{2pc}{1pc}{}}
\makeatother

  \usepackage{stmaryrd}
\usepackage{amsmath,amssymb,amsthm, mathrsfs, mathtools, bbold, mathbbol}
\usepackage{amscd,mathtools}
\usepackage{comment}
  \usepackage{geometry}
  \usepackage{graphicx,epstopdf}
  \usepackage{color,xcolor}
  \usepackage{enumerate}
  \usepackage{xspace}
  \usepackage{tipa}
  \usepackage[user,titleref]{zref}
 \usepackage{soul}
 \usepackage{xcolor}
  \usepackage{epsfig}
 \usepackage{needspace}
  
\usepackage{tikz}
\usetikzlibrary{shapes.geometric}
\usetikzlibrary{angles}

\usepackage{caption, subcaption}
\usepackage{tikz-cd}
\usetikzlibrary{calc,decorations.pathreplacing}
\usepackage{tikz}
\usetikzlibrary{calc}
\usetikzlibrary{arrows}
\usetikzlibrary{decorations.pathreplacing}
\usetikzlibrary{intersections}

\usepackage{mathrsfs}
\usetikzlibrary{matrix}
\usetikzlibrary{positioning}
\DeclareMathAlphabet{\pazocal}{OMS}{zplm}{m}{n}
\tikzset{>=stealth}

  \usepackage{hyperref}

  \hypersetup{
    colorlinks=true,
    linkcolor=blue,
    filecolor=magenta,      
    urlcolor=cyan,
    citecolor=purple,
    pdftitle={Overleaf Example},
    pdfpagemode=FullScreen,
    }

\usepackage{float}

  \newcommand{\EE}{\mathbb{E}}

  \newcommand{\ZZ}{\mathbb{Z}}

  \newcommand{\gothic}{\mathfrak}

  \newcommand{\go}{{\gothic o}}

  \newtheorem{theorem}{Theorem}[section]
  \newtheorem{proposition}[theorem]{Proposition}
  \newtheorem{corollary}[theorem]{Corollary}
  \newtheorem{lemma}[theorem]{Lemma}

  \newtheorem{claim}[theorem]{Claim}
  
  \newtheorem*{claim*}{Claim}

  \newtheorem{introthm}{Theorem}

  \theoremstyle{definition}
  \newtheorem{definition}[theorem]{Definition}
  
  \newtheorem{example}[theorem]{Example}
  
  \newtheorem*{question*}{Question}
  \newtheorem*{answer*}{Answer}
  \newtheorem*{application*}{Application}

  \theoremstyle{remark}
  
  \newtheorem*{remark*}{Remark}
  
  \newcommand{\Teich}{{Teichm\"uller }} 
  \newcommand{\sQ}{{\sf Q}}

  \newcommand{\qq}{{\sf q}}

\DeclareMathOperator{\Cay}{Cay}

  \newcommand{\param}{{\mathchoice{\mkern1mu\mbox{\raise2.2pt\hbox{$
  \centerdot$}}
  \mkern1mu}{\mkern1mu\mbox{\raise2.2pt\hbox{$\centerdot$}}\mkern1mu}{
  \mkern1.5mu\centerdot\mkern1.5mu}{\mkern1.5mu\centerdot\mkern1.5mu}}}

\DeclarePairedDelimiterX{\norm}[1]{\lvert}{\rvert}{#1}
\DeclarePairedDelimiterX{\Norm}[1]{\lVert}{\rVert}{#1}

  \renewcommand{\setminus}{{\smallsetminus}}

  \newcommand{\from}{\colon\thinspace}

\newcommand{\CAT}{\ensuremath{\operatorname{CAT}(0)}\xspace}
\renewcommand{\Cay}{\text{Cay}}
    
\title[First-passage percolation, non-positive curvature, and radial maps]{First-passage percolation, non-positive curvature, and radial maps}

\author{Dominic Bair}
 \address{Department of Mathematics,  University of Tennessee at Knoxville, Knoxville, TN, USA}
\email{dbair@vols.utk.edu}
\author{Sagnik Jana}
 \address{Department of Mathematics,  University of Tennessee at Knoxville, Knoxville, TN, USA}
\email{sjana1@vols.utk.edu}
\author{Yulan Qing}
 \address{Department of Mathematics,  University of Tennessee at Knoxville, Knoxville, TN, USA}
 \email{yqing@utk.edu}

\begin{document}
\begin{abstract}
Given an infinite connected graph $G$, a way to randomly perturb its metric is to assign random i.i.d. lengths to the edges of the graph, a process called first-passage percolation. Assume that the graph is infinite and of bounded degree. Assume the edge length distribution, $\nu$, has a finite expectation and is supported on $[0, \infty)$. We prove in this paper that non-positive curvature almost surely is not preserved by the associated percolation. In particular, Gromov hyperbolicity and coarse \CAT property of graphs are almost surely not preserved. We also show that if a graph contains a Morse geodesic ray, then the resulting image of the ray under first-passage percolation is no longer Morse. Lastly, we show that first-passage percolation almost surely is a radial map on $G$. 
\end{abstract}
\maketitle
\input{introduction}

\input{preliminaries}

\input{fpp_for_hyperbolic_graph}
\input{cat0_spaces_and_contracting_boundary}
\input{morse_sets_not_preserved}

\input{radial_map}

\bibliographystyle{alpha}

\end{document}

%% file: introduction.tex
\needspace{5\baselineskip}

\section{Introduction}

First-passage percolation (FPP) provides a probabilistic framework for studying random perturbations of a given geometric structure. In the classical model, one assigns independent and identically distributed (i.i.d.) random edge lengths to a fixed underlying graph. For more background, we refer to \cite{HW65, ADH15}.

To recall the setup, let $X$ be a connected, undirected graph with vertex set $V$ and edge set $E$. For each function $\omega : E \to (0, \infty)$, we define a weighted path metric $d_\omega$ on $V$ by assigning to each edge $e \in E$ the weight $\omega(e)$. For vertices $v_1, v_2 \in V$, the distance $d_\omega(v_1, v_2)$ is the infimum over all paths $\gamma = (e_1, \ldots, e_m)$ joining $v_1$ to $v_2$ of the length  
\[
|\gamma|_\omega = \sum_{i=1}^m \omega(e_i).
\]
When $\omega(e) \equiv 1$ for all $e \in E$, this reduces to the standard graph metric $d$.  
We equip the space of all weight assignments $\Omega = [0, \infty)^E$ with the product probability measure $P$, where the coordinate distributions are independent samples from a fixed measure $\nu$ supported on $[0, \infty)$.  

A long-standing open problem in first-passage percolation asks whether the model on $\mathbb{Z}^2$ admits, with positive probability, a bi-infinite geodesic. Kesten (Saint-Flour, 1984)(\cite{Kesten}) attributes this question to Furstenberg, and it remains unresolved despite partial progress by Licea and Newman~\cite{LN96}. Their work also notes that the conjecture—asserting the nonexistence of such geodesics—originated independently in the physics literature on spin glasses. Wehr and Woo~\cite{WW} later showed that in a half-plane model with continuous edge-length distributions of finite mean, no two-sided infinite geodesics can occur.

The situation contrasts sharply when we allow the space to have negative curvature.   Benjamini, Tessera, and Zeitouni~\cite{BZ12, BT17} established several key results and posed influential open questions. 
Notably,~\cite{BZ12} proved the tightness of passage-time fluctuations from the origin to the boundary of large balls, and~\cite{BT17} showed almost sure existence of bi-infinite geodesics in hyperbolic spaces. Further progress was made by Basu and Mahan~\cite{BM22}, who studied FPP on hyperbolic groups. Under appropriate assumptions on the edge-weight distribution, they demonstrated that the passage-time velocity exists generically with respect to the Patterson–Sullivan measure and is almost surely constant by ergodicity of the group action on $\partial G$. They also investigated coalescence phenomena and directional variance of passage times. However, one question remains open, are the bi-infinite geodesic lines in hyperbolic graphs after first-passage percolation still hyperbolic? In this paper we answer this question in the negative.

We fix some notation for paths in graphs. For a simple graph $X$, a path $\gamma = (e_1, \ldots, e_n)$ joining vertices $x$ and $y$ consists of consecutive edges connecting $\gamma(0) = x$ to $\gamma(n) = y$. The subpath $\gamma([i,j]) = (e_{i+1}, \ldots, e_j)$ joins $\gamma(i)$ to $\gamma(j)$, and infinite or bi-infinite paths are defined analogously.

\begin{introthm}\label{introthm1}
Let $X$ be a Gromov hyperbolic space, and an infinite, connected graph with bounded degree. Suppose for any bounded set $B$, at least one component of $X \setminus B$ is not a tree. Assume $\EE \omega_e < \infty$, $\text{Supp}(\nu) = (0,\infty)$, and $\nu(\{0\}) = 0$. Then, there exists a full measure subset \(\Omega_1\subset \Omega\) such that for all \(\omega\in \Omega_1\) $X_\omega$ is not Gromov hyperbolic.
\end{introthm}

Secondly, we show that even when the curvature is relaxed from coarsely negative to coarsely non-positive, this property is still not preserved by first-passage percolation.

\begin{introthm}\label{introthm2}
Let $X$ be an infinite, connected graph and suppose $X$ is quasi-isometric to a \CAT space. Suppose $X$ has bounded degree and is not a tree outside of any bounded set.  Assume $\EE \omega_e < \infty$ and $\nu(\left\{0\right\}) = 0$. Then, there exists a full measure subset $\Omega_2 \subset \Omega$ such that for all $\omega \in \Omega_2$, $X_\omega$ is not quasi-isometrically \CAT.

\end{introthm}

Lastly, we look more generally at Morse boundaries of proper geodesic graphs. We show that Morse rays do not map to Morse sets under first-passage percolation.

\begin{introthm}\label{introthm5}
Let $X$ be an infinite, connected graph with bounded degree and suppose $\alpha$ is a Morse geodesic ray such that $\alpha$ has infinitely many disjoint paths joining two vertices on $\alpha$. Assume $\EE \omega_e < \infty$ and $\nu(0) = 0$. Then for almost every $\omega$, $\omega(\alpha)$ admits no Morse gauge.
\end{introthm}

It is thus interesting to ask if there is a boundary at infinity that the first-passage percolation is likely to preserve. In \cite{JQ25a} the authors relaxed the condition offered by \cite{BT17} to a more relaxed condition which is the existence of a  bi-infinite sublinearly Morse lines. Sublinearly Morse boundary is a large set of directions whose Morse property relaxes as a sublinear function of the radius as a geodesic gravel towards infinity. It is constructed to produce a metrizable, quasi-isometrically invariant boundary that is often also a topological model for suitable random walks on the associated group  \cite{QRT1, QRT2, GQR22}.

Lastly, we identify a geometric property of first-passage percolation. A map between metric spaces is \emph{radial} if it maps geodesic segments that lie on geodesic rays emanating from a fixed basepoint, $\go$, to quasi-isometric segments. That is to say, it behaves like a quasi-isometry only along radial directions. We show that 
\begin{introthm}[Corollary~\ref{cor:radial}]
 
 Let $\go$ be a vertex of $X$. Then there exists a $c>0$ such that for $a.e.~\omega$, there exists $r_2 =r_2(\omega)$ such that FPP is a $(c,r_2)$-radial map with respect to $\go$.

\end{introthm}
We anticipate this result will help us understand how first-passage percolation changes the metric property on sublinearly Morse boundaries.

\subsection*{Open questions}
 Theorem A and Theorem C show that first-passage percolation does not preserve all phenomenon born from negative curvature, and Theorem B shows that it does not preserve curvatures bounded from above. We ask if it is more likely that first-passage percolation preserves a phenomenon born from positive curvature or curvatures bounded from above. Specifically, we ask the following. Let $X$ be an infinite graph with uniformly bounded degree:
\begin{itemize}
\item Let $X$ be quasi-isometric to an Alexandrov space, is  $X_\omega$ almost surely an Alexandrov space? 
\item Let $X$ be quasi-isometric to an injective space, is  $X_\omega$ almost surely an injective space?

\end{itemize}
These questions are not assuming the existence of bi-infinite geodesic line, which is also open for these two spaces. Lastly given the result of \cite{JQ25a}, it makes sense to ask:
\begin{itemize}
\item Is first-passage percolation likely to preserve sublinearly Morse boundaries?
\end{itemize}
We anticipate a positive answer in the upcoming work of \cite{JQ25b}. If the answer is positive, then given the positive result on radial map, and in 2015 it is shown \cite{DV15} that radial maps induces H\"{o}lder continuous map on Gromov boundaries, it is then natural to ask
\begin{itemize}
\item Does first-passage percolation almost surely induce a H\"older map on the sublinearly Morse boundary?
\end{itemize}

\subsection*{Acknowledgement} The third-named author would like to thank Gabriel Pallier and Romain Tessera for helpful discussions.

%% file: preliminaries.tex
\needspace{5\baselineskip}
\section{preliminaries}

\subsection{First-passage percolation} \label{sec23} \mbox{}\\
 $(X,d)$ be a connected graph with vertex set $V$ and edge set $E$. 
 with the graph distance metric $d$. Fix a probability measure $\nu$ supported on  $[0, \infty)$ and consider the product probability space
\[
\Omega = [0, \infty)^E, \quad \mathbb{P} = \nu^{\otimes E}.
\]
A typical element of $(\Omega, \mathbb{P})$ will be denoted by $\omega = \{\omega(e)\}_{e\in E}$. We define i.i.d random variables $X_e:\Omega \rightarrow [0, \infty)$ as $ X_e(\omega) = \omega(e) $ with law $\nu$. 
Assigning  $\omega(e)$ as the length of edge $e$ defines a random metric (apriori, a pseudometric) on $X$, known as the first-passage percolation (FPP) metric. The metric is defined as follows:

\begin{definition}

Let $\gamma = \{e_1, \dots, e_k\}$ be an edge path in $(X,d)$. For $\omega \in \Omega$, the $\omega$--length of $\gamma$ is defined as, 
\[
\bigr|\gamma\bigl|_{\omega}  = \sum_{e \in \gamma} \omega(e)
\]
The \emph{time coordinate of $\gamma$ under $\omega$} is defined as the random variable $T(\gamma,\omega)$ on $(\Omega,\mathbb P)$ and 
\[
T(\gamma,\omega) = \big|\gamma\big|_\omega.
\]

The \emph{FPP distance} is defined by, 
\[
d_\omega(x, y) = \inf_\gamma \bigl|\gamma\bigr|_\omega
\]
where the infimum is over all paths $\gamma$ with terminal vertices $x$ and $y$. A path realizing $d_\omega(x, y)$ is called a \emph{\( \omega \)-geodesic}. The \emph{first-passage time} between $x$ and $y$ is defined as the random variable $T(x,y)$ on $(\Omega,\mathbb P)$ and 
    \[
    T(x,y)(\omega) = d_\omega(x,y)
    \]
Further, if $\Gamma$ is a set of paths on $(X,d)$ with positive measure, the \emph{first-passage time of $\Gamma$ under $\omega$} is defined as
    \[
    T_\Gamma(\omega) = \inf_{\gamma\in\Gamma} T(\gamma,\omega). 
    \]
\end{definition}
    
\subsection{Geometric implication of first-passage percolation}

We end this preliminary section with a brief summary of known results on first-passage percolation required for the proofs that follow.

\begin{definition}
    [\cite{HW65}, 2.2] Two sets of paths $\Gamma_1$ and $\Gamma_2$ are said to be \emph{equivalent under lateral shift} is there exists a one-to-one mapping $f$ of $\Gamma_1$ onto $\Gamma_2$ such that for every subset of paths $\gamma_1,\gamma_2,\dots,\gamma_k$ belonging to $\Gamma_1$, the joint distribution $T(\gamma_1,\omega),T(\gamma_2,\omega),\dots ,T(\gamma_k,\omega)$ is identical with the joint distribution of $T(f(\gamma_1),\omega),T(f(\gamma_2),\omega),\dots ,T(f(\gamma_k),\omega)$.
\end{definition}

\begin{theorem}
    [\cite{HW65}, Theorem 2.2.1] \label{theorem:lateral_shift} If $\Gamma_1$ and $\Gamma_2$ are sets of paths equivalent under lateral shift, then $T_{\Gamma_1}(\omega)$ and $T_{\Gamma_2}(\omega)$ are identically distributed random variables.
\end{theorem}

\begin{example}
    (See Section \ref{sec:cay} for more background.) Let $G$ be an infinite finitely generated group with finite symmetric generating set $S$. Let $X = \Cay(G,S)$. Let $\Gamma_1$ be the set of all paths connecting $e$ to $g\in G$. Let $h\in G$. Then let $\Gamma_2$ be the set of all paths connecting $h$ to $hg$. Since $G$ acts geometrically on $X$ by left action, we have $h\Gamma_1 = \Gamma_2$. Hence, $\Gamma_1$ and $\Gamma_2$ are equivalent under lateral shift.
\end{example}
 
\begin{proposition}[\cite{BT17}, Lemma 2.3]\label{prop:upper_bnd}\mbox{}\\ Let $X$ be a connected graph, and let $\gamma$ be a self avoiding path. Assume $0<b=\mathbb{E} \omega_{e}<\infty$. Then for a.e. $\omega$, there exists $r_{0}=r_{0}(\omega)$ such that for all $i \leq 0 \leq j$,

\[
|\gamma([i, j])|_{\omega} \leq 2 b(j-i)+r_{0}
\]
\end{proposition}

\begin{proposition}[Lemma 2.4 \cite{BT17}] \label{shrinkingprop} 
Let $X$ be an infinite connected graph with bounded degree and assume that $\nu({0})=0.$ There exists an increasing function $\alpha:(0,\infty)\to (0,1]$ such that $\lim_{t\to0}\alpha(t) = 0,$ and such that for all finite path $\gamma$ and all $\epsilon>0,$
$$\mathbb{P}(|\gamma|_\omega \leq \epsilon|\gamma|) \leq \alpha(\epsilon)^{|\gamma|}.$$
\end{proposition}

\begin{proposition} [\cite{BT17}, Lemma 2.5]\label{prop:lower_bound} Let $X$ be an infinite connected graph with bounded degree, and let $o$ be some vertex of $X$. Assume $\nu(\{0\})=0$. Then there exists $c>0$ such that for a.e. $\omega$, there exists $r_{1}=r_{1}(\omega)$ such that for all finite path $\gamma$ such that $d(\gamma, o) \leq |\gamma|$, one has

\[|\gamma|_{\omega} \geq c|\gamma|-r_{1}\]

\end{proposition}

\subsection{Basic probabilistic tools }
For the reader’s convenience, we collect in this section a few basic results that will be used in the sequel. While these are standard, their inclusion makes the paper more self-contained.

\begin{lemma}[Borel--Cantelli]\label{lem:b-c}
Let $(A_n)_{n\ge1}$ be a sequence of events in a probability space
$(\Omega,\mathcal{F},\mathbb{P})$.
\begin{enumerate}
    \item[\textnormal{(1)}] 
    If $\displaystyle\sum_{n=1}^{\infty} \mathbb{P}(A_n) < \infty$, then
    \[
    \mathbb{P}\bigl(A_n \text{ i.o.}\bigr)
    = \mathbb{P}\Bigl(\limsup_{n\to\infty} A_n\Bigr) = 0.
    \]
    \item[\textnormal{(2)}] 
    If the events $(A_n)$ are independent and 
    $\displaystyle\sum_{n=1}^{\infty} \mathbb{P}(A_n) = \infty$, then
    \[
    \mathbb{P}\bigl(A_n \text{ i.o.}\bigr) = 1.
    \]
\end{enumerate}
\end{lemma}
The second part of this lemma will play a central role in several arguments throughout the paper.

\begin{definition}
    Let $X$ be a random variable. The \emph{distribution function} of $X$ is defined as
    \[F_X(x) = \mathbb{P}(X\leq x).\]
\end{definition}

\subsection{Background in coarse geometry}

\begin{definition}[Quasi-isometric embedding] \label{Def:Quasi-Isometry} 
Let $(X , d_X)$ and $(Y , d_Y)$ be metric spaces. For constants $q \geq 1$ and
$Q \geq 0$, we say a map $f \from X \to Y$ is a 
$(q, Q)$-\textit{quasi-isometric embedding} if, for all points $x_1, x_2 \in X$
$$
\frac{1}{q} d_X (x_1, x_2) - Q  \leq d_Y \big(f (x_1), f (x_2)\big) 
   \leq q \, d_X (x_1, x_2) + Q.
$$
If, in addition, every point in $Y$ lies in the $Q$-neighborhood of the image of 
$f$, then $f$ is called a $(q, Q)$-quasi-isometry. When such a map exists, $X$ 
and $Y$ are said to be \textit{quasi-isometric}. 

A quasi-isometric embedding $f^{-1} \from Y \to X$ is called a \emph{quasi-inverse} of 
$f$ if for every $x \in X$, $d_X(x, f^{-1}f(x))$ is uniformly bounded above. 
In fact, after replacing $q$ and $Q$ with larger constants, we assume that 
$f^{-1}$ is also a $(q, Q)$-quasi-isometric embedding, 
\[
\forall x \in X \quad d_X\big(x, f^{-1}f(x)\big) \leq Q \qquad\text{and}\qquad
\forall y \in Y \quad d_Y\big(y, f\,f^{-1}(x)\big) \leq Q.
\]
\end{definition}

A \emph{geodesic ray} in $X$ is an isometric embedding $\beta \from [0, \infty) \to X$. We fix a base-point $\go \in X$ and always assume that $\beta(0) = \go$, that is, a geodesic ray is always assumed to start from this fixed base-point. 
\begin{definition}[Quasi-geodesics] \label{Def:Quadi-Geodesic} 
In this paper, a \emph{quasi-geodesic ray} is a continuous quasi-isometric 
embedding $\beta \from [0, \infty) \to X$  starting from the basepoint $\go$. 
\end{definition}
The additional assumption that quasi--geodesics are continuous is not necessary for the results in this paper to hold, but it is added for convenience and to make the exposition simpler. 

If $\beta \from [0,\infty) \to X$ is a $(\qq, \sQ)$-quasi-isometric embedding, and $f \from X \to Y$ is a $(q, Q)$-quasi-isometry then the composition  $f \circ \beta \from [t_{1}, t_{2}] \to Y$ is a quasi-isometric embedding, but it may not be continuous. However, one can adjust the map slightly to make it continuous (see Definition 2.2 \cite{QRT1}) such that $f \circ \beta$ is a $(q\qq, 2(q\qq + q \sQ + Q))$-quasi-geodesic ray.

Similar to above, a \emph{geodesic segment} is an isometric embedding 
$\beta \from [t_{1}, t_{2}] \to X$ and a \emph{quasi-geodesic segment} is a continuous 
quasi-isometric embedding \[\beta \from [t_{1}, t_{2}] \to X.\] 


Note that, if a segment is presented without a subscript, for example $[y_{1}, y_{2}]$, then it is a geodesic segment between the two points.
Let $\beta$ be a quasi-geodesic ray. Define 
\[
\Norm{x} : = d(\go, x).
\]

\begin{definition}[Morse property] \label{Def:Morse}
Let $(X,d)$ be a geodesic metric space. A set $\Gamma\subseteq X$ is called \emph{Morse} if it satisfies the following property:

For every $q \ge 1$ and $Q \ge 0$, there exists a constant $M = M(q,Q)$, called a \emph{Morse gauge}, such that every $(q,Q)$--quasi-geodesic $\gamma$ with endpoints on $\Gamma$ stays within the $M$-neighborhood of $\gamma$; that is,
\[
  \gamma \subseteq N_M(\Gamma).
\]

In particular, a quasi-geodesic is called \emph{Morse} if it admits a Morse gauge.

\end{definition}


\begin{definition} \label{Def:hyperbolic}\cite{gromov}
Let $(X,d)$ be a geodesic metric space and let $\delta \ge 0$.
\begin{enumerate}
    \item For any $x,y,z \in X$, a geodesic triangle $\triangle xyz$ is said to be \emph{$\delta$-slim} if each side of the triangle is contained in the $\delta$--neighborhood of the union of the other two sides; that is,
    \[
    [x,y] \subseteq N_\delta([y,z] \cup [x,z])
    \]
    and similarly for the other sides.
    
    \item The space $(X,d)$ is called \emph{$\delta$-hyperbolic} (or \emph{Gromov-hyperbolic}) if every geodesic triangle in $X$is $\delta$-slim.
\end{enumerate}
A space is \emph{hyperbolic} if it is $\delta$-hyperbolic for some
$\delta \ge 0$.
\end{definition}

\subsection{Cayley graphs and group actions}\label{sec:cay}
\begin{definition}
Let $G$ be a group and let $S \subseteq G$ be a symmetric generating set, i.e. $s\in S$ if and only if $s^{-1}\in S$. The \emph{Cayley graph} $X = \Cay(G,S)$ is the undirected graph defined as follows:
\begin{enumerate}
    \item The vertex set is 
    \[
        V(X) = \{ g \mid g \in G \}.
    \]
    \item The edge set is 
    \[
        E(X) = \bigl\{ \{g,h\} \mid g,h \in G,~ g^{-1}h \in S \bigr\}.
    \]
\end{enumerate}
Henceforth, we will assume \emph{every} graph, Cayley or otherwise, is \emph{undirected}.
\end{definition}

\begin{definition}
Let $G$ be a group and $X$ a set.  
A \emph{(left) action} of $G$ on $X$, $G \curvearrowright X$, is a map
\[
\cdot \colon G \times X \longrightarrow X, \quad (g,x) \mapsto g \cdot x,
\]
such that for all $g,h \in G$ and $x \in X$:
\begin{enumerate}
    \item $e \cdot x = x$, where $e$ is the identity of $G$,
    \item $g \cdot (h \cdot x) = (gh) \cdot x$.
\end{enumerate}
\end{definition}


\begin{definition}
Let $G$ be a group acting by homeomorphisms (or isometries) on a topological
(or metric) space $X$.
\begin{enumerate}
    \item The action is \emph{free} if $g \cdot x = x$ for some $x \in X$
    implies $g = e$.
    \item The action is \emph{properly discontinuous} if for every compact set
    $K \subseteq X$, the set
    \[
    \{ g \in G : gK \cap K \neq \emptyset \}
    \]
    is finite.
    \item The action is \emph{cocompact} if the quotient space $X/G$
    is compact.
\end{enumerate}
\end{definition}


\begin{example}
Let $X = \Cay(G,S)$ be the Cayley graph of a finitely generated group $G$
with finite symmetric generating set $S$.  
The left-regular action
\[
g \cdot h = gh, \qquad g,h \in G,
\]
induces an action $G \curvearrowright X$ by graph automorphisms.  
This action is free, properly discontinuous, and cocompact; in particular, it is a \emph{geometric action}.
\end{example}


\subsection*{ \CAT spaces}
A proper geodesic metric space $(X, d_X)$ is \CAT if geodesic triangles in $X$ are at 
least as thin as triangles in Euclidean space with the same side lengths. To be precise, for any 
given geodesic triangle $\triangle pqr$, consider the unique triangle 
$\triangle \overline p \overline q \overline r$ in the Euclidean plane with the same side 
lengths. For any pair of points $x, y$ on edges $[p,q]$ and $[p, r]$ of the 
triangle $\triangle pqr$, if we choose points $\overline x$ and $\overline y$  on 
edges $[\overline p, \overline q]$ and $[\overline p, \overline r]$ of 
the triangle $\triangle \overline p \overline q \overline r$ so that 
$d_X(p,x) = d_\EE(\overline p, \overline x)$ and 
$d_X(p,y) = d_\EE(\overline p, \overline y)$ then,
\[ 
d_{X} (x, y) \leq d_{\EE^{2}}(\overline x, \overline y).
\] 

\begin{figure}[h]
    \centering
    \includegraphics[width=0.6\linewidth]{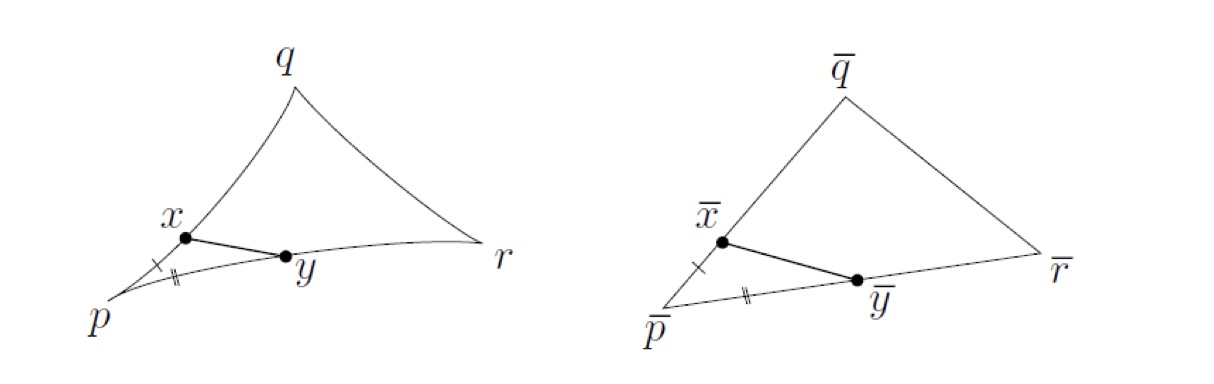}
            \caption{ \cite{BH1} Non-fat triangles: The distance between the points x and y is no greater
than the distance between their comparison points
$\overline x$ and $\overline y$.}
            \label{fig:non-fat}
        \end{figure}
 A metric space $X$ is {\it proper} if closed metric balls are compact. For the remainder of the paper, we assume $X$ is a proper \CAT space. Here, we list some properties of proper \CAT spaces that are needed later (see 
 \cite{BH1}). 

\begin{lemma} \label{Lem:CAT} 
 A proper \CAT space $X$ has the following properties:
\begin{enumerate}[i.]
\item It is uniquely geodesic, that is, for any two points $x, y$ in $X$, 
there exists exactly one geodesic connecting them. Furthermore, $X$ is contractible 
via geodesic retraction to a base point in the space. 
\item The nearest-point projection from a point $x$ to a geodesic line $b$
is a unique point denoted $x_b$. In fact, the closest-point projection map
\[
\pi_b \from X \to b
\]
is Lipschitz. 
\end{enumerate}
\end{lemma}

%% file: fpp_for_hyperbolic_graph.tex

\needspace{5\baselineskip}
\section{First-passage percolation for hyperbolic graphs}\label{sec:fpp-hyper}
In this section we will discuss if hyperbolicity is preserved under first passage percolations. Recall the definition of Gromov hyperbolicity in Defintion~\ref{Def:hyperbolic}, classical examples include trees (which are $0$-hyperbolic), the hyperbolic plane 
$\mathbb{H}^2$. In this paper we examine the Cayley graphs of word-hyperbolic groups, which are all Gromov hyperbolic. 
Gromov hyperbolicity is a quasi-isometry invariant: if two metric spaces are quasi-isometric 
and one is Gromov hyperbolic, so is the other.

A technical point arises when discussing hyperbolicity in the context of first-passage percolation. Let \(X\) be a connected graph, and assign to each edge \(e\) a nonnegative length \(\omega(e)\), interpreted as the passage time across \(e\). While this vertex-based metric suffices for many probabilistic questions, it is not the most natural framework for geometric considerations, such as Gromov hyperbolicity. Indeed, the definition of \(\delta\)-hyperbolicity involves geodesics and their subdivisions (e.g.\ midpoints), and these need not be realized at vertices of \(X\). For this reason, we pass to the \emph{metric realization of the graph \(X_\omega\)}: each edge \(e\) is replaced by a genuine interval of length \(\omega(e)\), and we consider the resulting random metric space. Formally, for each edge $e=(v_1,v_2)\in E$, we glue the endpoints of an interval of length $\omega(e)$ to $v_1$ and $v_2$. In this \emph{metric graph}, geodesics are continuous paths parameterized by arc-length, and notions such as midpoints are always well-defined. Crucially, this construction preserves the $\omega$-distance between vertices while providing a setting in which the standard definition of \(\delta\)-hyperbolicity applies without artificial restrictions to the discrete vertex set. For Sections \ref{sec:fpp-hyper}, \ref{sec:fpp_cat0}, and \ref{sec:fpp_Morse}, $X_\omega$ shall represent the metric realization of $X$ after percolation.

We first introduce the following example, which captures the proof idea for general cases.
\begin{example}
    Suppose $X$ is the Cayley graph of $G =\mathbb{Z}/2\mathbb{Z}*\mathbb{Z}/3\mathbb{Z}$ with presentation $G =\langle a,b~|~a^2,b^3\rangle$. See Figure \ref{fig:non-hyper}. Then $X$ contains countably many copies of a non-degenerate triangle (that is, any two edges do not contain the third edge) with 3 edges making up the geodesics of the triangle. Call these triangles $T_n$ with edges $e_{n,1},e_{n_2},e_{n,3}$. Note $X$ is 1-hyperbolic. Suppose $\text{Supp}(\omega_e) = (0,\infty)$. We claim that for $a.e. ~\omega,$ the metric realization of $X_\omega$ is not $\delta$-hyperbolic for any $\delta\geq 0$.
    \vspace{1cm}
    
     \begin{figure}
    \centering
    \includegraphics[width=0.5\linewidth]{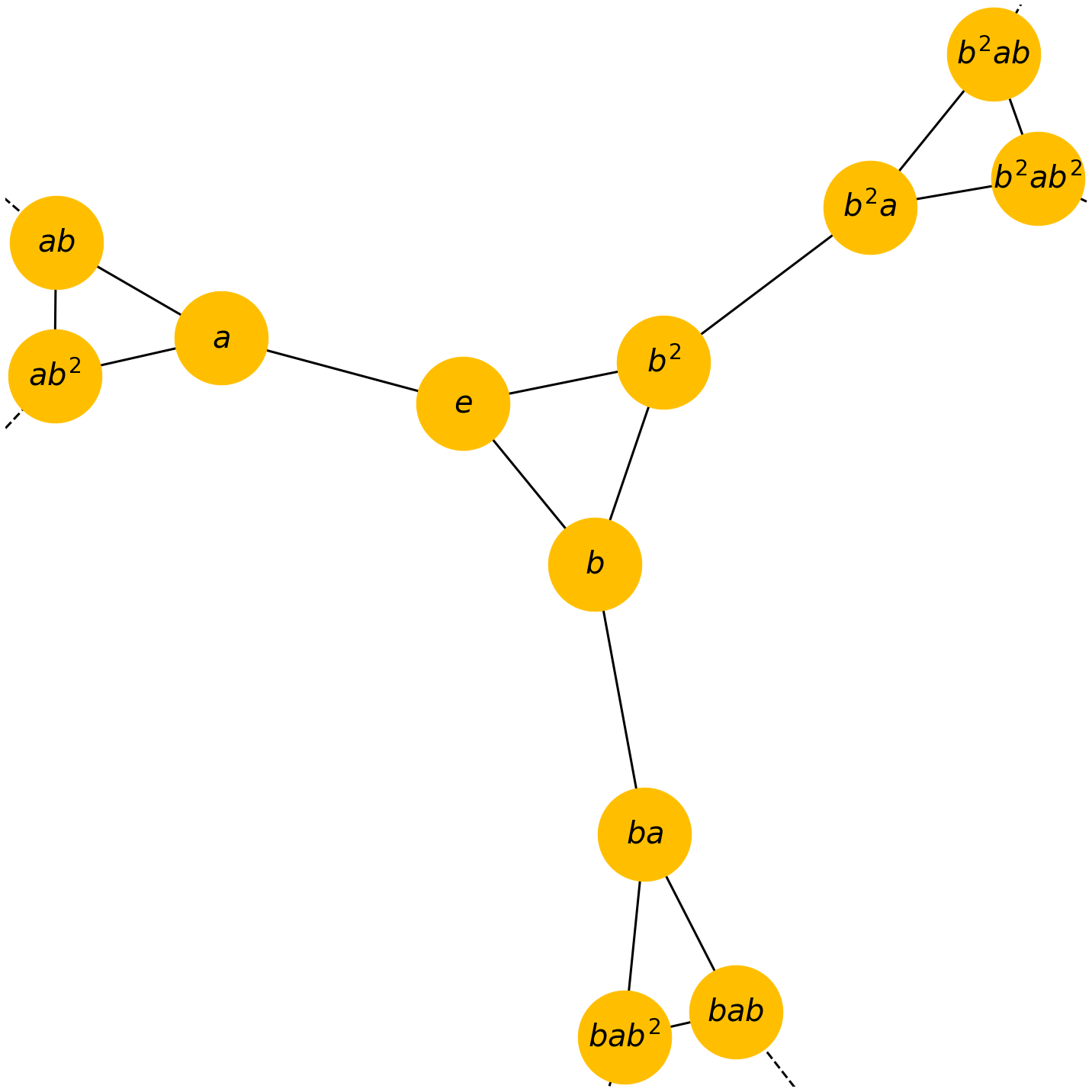}
            \caption{$\Cay(\mathbb{Z}/2\mathbb{Z}*\mathbb{Z}/3\mathbb{Z}, \langle a,b ~\big|~ a^2,b^3\rangle)$.}
            \label{fig:non-hyper}
        \end{figure}
        
    \begin{proof}
        Suppose by way of contradiction that $X_\omega$ is $\delta$-hyperbolic for some $\delta\geq 0$. We note that after percolation, the edges of each $T_n$ may not be geodesics, so we compute the probability that a given $\omega(T_n)$ has each side length greater than $ 3\delta$ and less than $ 4\delta$. This event, call it $A_n$, corresponds to $\omega(T_n)$ not being $\delta$-hyperbolic, but also each edge of $\omega(T_n)$ being a geodesic. That is to say $\omega(T_n)$ is a non-degenerate $\omega$-triangle which is not $\delta$-slim. See Figure \ref{fig:Tn}. The probability of $A_n$ occurring is given by

        \begin{figure}
            \centering
            \tikzset{every picture/.style={line width=0.75pt}} 

\begin{tikzpicture}[x=0.75pt,y=0.75pt,yscale=-1,xscale=1]

\draw   (339.71,100.5) -- (415.58,192) -- (263.83,192) -- cycle ;
\draw  [color={rgb, 255:red, 208; green, 2; blue, 27 }  ,draw opacity=1 ][dash pattern={on 4.5pt off 4.5pt}] (242.04,216.56) .. controls (228.41,203.63) and (242.04,165.81) .. (272.47,132.08) .. controls (302.91,98.34) and (338.63,81.48) .. (352.26,94.41) .. controls (365.89,107.34) and (352.26,145.16) .. (321.83,178.89) .. controls (291.39,212.63) and (255.67,229.49) .. (242.04,216.56) -- cycle ;
\draw  [color={rgb, 255:red, 208; green, 2; blue, 27 }  ,draw opacity=1 ][dash pattern={on 4.5pt off 4.5pt}] (324.31,93.93) .. controls (339.43,81.07) and (376.32,98.21) .. (406.7,132.2) .. controls (437.08,166.19) and (449.45,204.17) .. (434.33,217.02) .. controls (419.2,229.87) and (382.31,212.74) .. (351.93,178.75) .. controls (321.55,144.76) and (309.18,106.78) .. (324.31,93.93) -- cycle ;
\draw [color={rgb, 255:red, 208; green, 2; blue, 27 }  ,draw opacity=1 ][fill={rgb, 255:red, 224; green, 30; blue, 54 }  ,fill opacity=1 ]   (384.58,153.75) -- (409.08,135.5) ;
\draw [color={rgb, 255:red, 208; green, 2; blue, 27 }  ,draw opacity=1 ]   (297.33,152) -- (273.75,130.83) ;

\draw (266.58,90.15) node [anchor=north west][inner sep=0.75pt]  [font=\footnotesize]  {$T_{n}$};
\draw (398.83,148.02) node [anchor=north west][inner sep=0.75pt]  [font=\scriptsize,color={rgb, 255:red, 208; green, 2; blue, 27 }  ,opacity=1 ]  {$\delta $};
\draw (274.5,145.36) node [anchor=north west][inner sep=0.75pt]  [font=\scriptsize,color={rgb, 255:red, 208; green, 2; blue, 27 }  ,opacity=1 ]  {$\delta $};
\draw (331.33,179.57) node [anchor=north west][inner sep=0.75pt]  [font=\scriptsize]  {$e_{2}$};
\draw (304.67,145.23) node [anchor=north west][inner sep=0.75pt]  [font=\scriptsize]  {$e_{3}$};
\draw (364.67,144.9) node [anchor=north west][inner sep=0.75pt]  [font=\scriptsize]  {$e_{1}$};
\draw (329.33,200.9) node [anchor=north west][inner sep=0.75pt]  [font=\scriptsize]  {$3\delta $};
\draw (290,118.57) node [anchor=north west][inner sep=0.75pt]  [font=\scriptsize]  {$3\delta $};
\draw (368.67,122.23) node [anchor=north west][inner sep=0.75pt]  [font=\scriptsize]  {$3\delta $};

\end{tikzpicture}

            \caption{Diagram showing that $T_n$ having side $\omega$-lengths $3\delta$ is not $\delta$-hyperbolic, but is an $\omega$--triangle.}
            \label{fig:Tn}
        \end{figure}
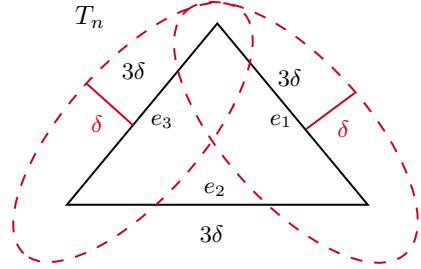
        \begin{align*}
            \mathbb{P}(3\delta &< \omega(e_{n,1}) <4\delta,~3\delta < \omega(e_{n,2}) <4\delta,~3\delta < \omega(e_{n,3}) <4\delta)\\
            &= \mathbb{P}(3\delta < \omega(e_{n,1}) \leq 4\delta)^3\quad(\text{Since $\omega_e$ are iid.})\\
            &= (F_{\omega_e}(4\delta) - F_{\omega_e}(3\delta))^3=\epsilon^3>0.
        \end{align*}
        This $\epsilon$ exists by the support of $\omega_e$ being $(0,\infty)$.
        \end{proof}
        Since each $T_n$ differs only by translation by group action, the distribution of each $A_n$ is i.i.d. In particular, each $T_n$ is equivalent under lateral shift. Then 
        \[
        \sum_{n=1}^\infty P(3\delta < \omega(e_{n,1}))^3 = \sum_{n=1}^\infty \epsilon^3 = \infty.
        \]
        So by the second part of the Borel--Cantelli lemma (\ref{lem:b-c}), 
        $$\mathbb{P}(A_n~\text{i.o.}) = 1.$$
        Thus, with probability 1, infinitely many $\omega(T_n)$ will not be $\delta$-hyperbolic. This contradicts $X_\omega$ being $\delta$-hyperbolic.

\end{example}


\begin{definition}
    We call $\Delta_\omega abc$ a,  $\omega$-geodesic triangle in $\left(X,d_\omega\right)$ with vertices $a,b,$ and $c$ with edges being the $\omega$-geodesics $[a,b]_\omega, [b,c]_\omega,$ and $[c,a]_\omega$.
\end{definition}
Note that $\Delta abc$ and $\Delta_\omega abc$ have the same (triangle) vertices; however, their edges and (graph) vertices along each edge need not agree.

\begin{proposition}[FPP maps cycles to cycles]\label{claim:cycletocycle} Let $X$ be a graph that contains a $k$-cycle $\gamma: [0,k]\to X$ with vertices $v_1,v_2,\dots,v_k,v_1$ such that $\gamma(i)=v_{i+1}$ for $i\in \ZZ/k\ZZ$. Suppose Supp$(\nu) = (0,\infty)$ and $\nu(\{0\}) = 0$. Then for $a.e.$ $\omega$, $X_\omega$, $\omega(\gamma)$ is a cycle of finite $\omega$-length in $X_\omega$ with vertices  $v_1,v_2,\dots,v_k,v_1$.
\begin{proof}
    \[\bigl|\gamma\bigr|_\omega = \sum_{i=1}^{k+1} \omega(e_i) <\infty~\text{a.s.}\]
    as $\omega(e_i)$ is almost surely finite for each $1\leq i\leq k+1$.
    Furthermore, $\omega(\gamma(i)) = \omega(v_{i+1}) = v_{i+1} \in X_\omega$ for all $1\leq i\leq k+1$. Therefore, $\omega(\gamma)$ a.s. has finite length with vertices  $v_1,v_2,\dots,v_k,v_1$.
    
    Lastly, by assumption, $\nu(\{0\}) = 0$. Thus, for $a.e.~\omega$, $\bigl|e_i\bigr|_\omega > 0$ for all edges $e_i$ in the $k$-cycle. Therefore, with probability 1, the cycle does not collapse to a point.
\end{proof}
\end{proposition}

\subsection{Proof of Theorem \ref{introthm1}}
To prove the main theorem, let us consider the following proposition,
\begin{proposition} \label{prop:nondslim}
    Suppose $0<\delta< \infty$. Suppose $X$ contains a $k$-cycle with vertices $v_1,v_2,\dots,v_k,v_1$. Further, suppose the image of this cycle contains an edge $e_i = [v_i,v_{i+1}]$ satisfying $\bigl|e_i\bigr|_\omega \geq 4\delta$. Then $X_\omega$ contains a triangle that is not $\delta$-slim.
    \begin{proof}
    \textbf{Case 1:} Suppose $\omega(e_i)$ is not an $\omega$-geodesic connecting $v_i$ to $v_{i+1}$. Then there exists some $\omega$-geodesic path in $X_\omega$ from $v_i$ to $v_{i+1}$ with $\omega$-length $\leq 4\delta$ that does not contain $\omega(e_i)$. Call this path $P$. Observe that the concatenation of $\omega(e_i)$ with $P$ is a cycle in $X_\omega$ and $\bigl|P\bigr|_\omega \leq \bigl|e_i\bigr|_\omega$. Let $a$ be the midpoint of $P$. Let $m$ be the midpoint of $\omega(e_i)$. Let $b \in [m,v_{i}]_\omega \subset \omega(e_i)$ be the point satisfying $d_\omega(m,b) = \delta$ Let $c\in [m,v_{i+1}]_\omega \subset \omega(e_i)$ be the point satisfying $d_\omega(m,c) = \delta$. Then $d_\omega(b,v_i) \geq \delta$, similarly $d_\omega(c,v_{i+1}) \geq \delta$. Further, the segment $[b,c]_\omega$ on $\omega(e_i)$ is an $\omega$-geodesic segment as $ d_\omega(b,c) = 2\delta \leq \delta + \bigl|P\bigr|_\omega + \delta \leq d_\omega(b,v_i) + \bigl|P\bigr|_\omega + d_\omega(v_{i+1},c)$. We need not consider other paths as $P$ is an $\omega$-geodesic path connecting $v_i$ to $v_{i+1}$.

    Then the path $[a,v_i]_\omega$ along $P$ concatenated with $[v_i,b]_\omega$ along $\omega(e_i)$ is an $\omega$-geodesic connecting $a$ and $b$. Similarly, $[a,v_{i+1}]_\omega$ along $P$ concatenated with $[v_{i+1},c]$ along $\omega(e_i)$ is an $\omega$-geodesic connecting $a$ to $c$. So $\Delta_\omega abc$ is a non-degenerate $\omega$-geodesic triangle in $X_\omega$. However, $N_\delta([a,b]_\omega) \cup N_\delta([a,c]_\omega)$ does not contain $m \in [b,c]_\omega$. Thus $\Delta_\omega abc$ is not $\delta$-slim.

    \textbf{Case 2:} Suppose $\omega(e_i)$ is an $\omega$-geodesic connecting $v_i$ to $v_{i+1}$. Recall that the concatenation $\omega(e_1),\dots,\omega(e_i),\dots,\omega(e_k)$ is a cycle in $X_\omega$. Denote the path $\omega(e_{i+1}),\omega(e_{i+2}),\dots,\omega(e_{i-1})$ as $P$. That is all edges of the cycle not including $\omega(e_i)$. Then $P$ is a path connecting $v_i$ to $v_{i+1}$. Since $\omega(e_i)$ is a geodesic connecting $v_i$ to $v_{i+1}$, $\bigl|P\bigr|_\omega \geq \bigl|e_i\bigr|_\omega \geq 4\delta$. Then there exists some point $q \in P$ satisfying $d_\omega(v_i,q) \geq 2\delta$, $d_\omega(q,v_{i+1}) \geq 2\delta$, and $d_\omega(q,v_i) = d_\omega(q,v_{i+1})$. This point $q$ exists as $\bigl|P\bigr|_\omega \geq 4\delta$ and for $a.e.~\omega$, $X_\omega$ is a geodesic metric space; thus for $a.e.~\omega$, $d_\omega:X_\omega \times X_\omega \to \mathbb{R}$ is well defined and continuous. Let $a = q$, $b = v_i$, and $c = v_{i+1}$. Then $\Delta_\omega abc$ is a non-degenerate $\omega$-geodesic triangle. Further, $N_\delta([a,b]_\omega) \cup N_\delta([a,c]_\omega)$ does not contain the midpoint of $[b,c]_\omega$, so $\Delta_\omega abc$ is not $\delta$-slim.
    \end{proof}
\end{proposition}

\begin{theorem} \label{theorem:A}
Let $X$ be a graph with bounded degree with the assumption that, for every bounded set $B$, at least one component of $X \setminus B$ is not a tree. Suppose Supp$(\nu)=(0,\infty)$ and $\nu(\{0\}) = 0$, further assume that $\EE \omega_e < \infty$. Then for $a.e.$ $\omega$, the metric realization of $X_\omega$ is not Gromov hyperbolic.

\begin{proof}Suppose, by way of contradiction, that $X_\omega$ is $\delta$-slim for some $\delta<\infty$.
From the assumption, it is clear that there exist infinitely many disjoint cycles, namely \{$C_n: n\in \mathbb{N}$\}. Suppose $E'\subset E $ is the set of all edges in {$C_n$} for each $n$. Since $\nu(\{0\}) = 0$ and $\nu$ has unbounded support, we have $\mathbb{P}(\omega(e_n) \geq 4\delta) = \epsilon$ for some $\epsilon>0$. Thus $\sum_{n=1}^\infty \mathbb{P}(\omega(e_n) \geq 4\delta) = \sum_{n=1}^\infty \epsilon = \infty$. Since the $\omega$-length of each edge is independent, by the second part of the Borel--Canelli Lemma, $\mathbb{P}(\omega(e_n) \geq 4\delta ~i.o.) = 1$ for every $e_n \in E'$. Thus, with probability 1, there exist infinitely many cycles $K_n$ $\subset$ \{$C_n: n\in \mathbb{N}$\} such that in $X_\omega$, $|e_n|_\omega\geq4\delta$ for some $e_n \in K_n$. Now by Proposition ~\ref{prop:nondslim} for $a.e.$ $\omega$, $X_\omega$ contains infinitely many triangles that are not $\delta$-slim. This contradicts the assumption that $X_\omega$ is $\delta$-hyperbolic. 
\end{proof}

\end{theorem}
From the above theorem, the following corollary about infinite hyperbolic groups follows directly.
\begin{corollary}
    Suppose $G$ is a finitely generated, infinite hyperbolic group. Let $S$ be a generating set and let $X= \Cay(G, S)$ with generating set $S$ chosen such that $X$ is not a tree. Suppose $\nu$ has support $(0,\infty)$. Then for a.e. $\omega$ the metric realization of $X_\omega$ is not Gromov hyperbolic.  
\end{corollary}
\begin{proof}
Since $X$ is not a tree, there exists at least one cycle. Let $C$ be the cycle of minimal length $k$ and diameter $r$. Since $G$ $\curvearrowright$ $X$ is properly discontinuous and cocompact, and in fact, geometric, it sends $C$ to the cycle $g_1\cdot C$ of the same length for some $g_1 \in G$. Suppose $U_1$ is the closed neighborhood of $V(g_1\cdot C)$ of radius $r+1$, where $V(g_1\cdot C)$ is the vertex set of $g_1\cdot C$. Since $X$ is infinite, there exists $g_2 \notin U_1$ and the cycle $g_2 \cdot C$ is disjoint from $g_1 \cdot C$. By repeating this argument, we get infinitely many disjoint cycles of size $k$. Now, from the assumption, suppose that $X_\omega$ is $\delta$-hyperbolic space for some $\delta>0$. Then, by Theorem~\ref {theorem:A} for a.e. $\omega,$ $X_\omega$ contains a triangle that is not $\delta$--slim. 
\end{proof}

%% file: cat0_spaces_and_contracting_boundary.tex
\needspace{5\baselineskip}
\subsection{Coarse \CAT spaces and first passage percolations} \label{sec:fpp_cat0}
In this section, we observe that the proof technique from the preceding section can be easily adapted to other situations where the geometry comes from controlling the thinness of triangles. One of these geometries is the \CAT geometry. 

Let $G$ be a finitely generated group. We say $G$ is \CAT if the group acts cocompactly, properly discontinuously, and by isometries on a \CAT space. Examples of such groups include right-angled Artin groups, right-angled Coxeter groups, and groups acting geometrically on spaces of strictly and uniformly negative curvature, such as surface groups. \CAT groups can also be formed via free product and the free amalgamated product of \CAT groups with suitable gluing. We note that Cayley graphs of \CAT groups are quasi-isometric to a \CAT space.
\begin{theorem}
Let $X$ be quasi-isometric to a \CAT space and suppose that there does not exist a bounded ball $B \subset X$ such that $X \setminus B$ is acyclic.  Then for a.e. $\omega$, $X_\omega$ is not quasi-isometric to any \CAT space.
\end{theorem}
\begin{proof}
Since there does not exist a bounded ball $B \subset X$ such that $X \setminus B$ is acyclic. $X$ contains infinitely many disjoint cycles. Let $\{C_i\}$ be an infinite collection of such cycles and let $x_i \in C_i$ be an edge from each cycle, respectively. Suppose the space is $(q, Q)$-quasi-isometric to a \CAT space, then every cycle can be of length at most $2q+Q$. Therefore, if any edge in a cycle has length more than  $2q+Q$, it is not 
$(q, Q)$-quasi--isometric to a \CAT space. Thus, we need all edges to have $\omega$-length shorter than $2q+Q$.
Since $\nu(\{0\}) = 0$ and $\nu$ has unbounded support, therefore we have $\mathbb{P}(\omega(e_n) \geq 2q+Q) = \epsilon$ for some $\epsilon>0$. Thus $\sum_{n=1}^\infty \mathbb{P}(\omega(e_n) \geq 2q+Q) = \sum_{n=1}^\infty \epsilon = \infty$. Since the $\omega$-length of each edge is independent, by the second part of the Borel--Cantelli Lemma ~\ref{lem:b-c}, $\mathbb{P}(\omega(e_n) \geq 2q+Q ~\text{i.o.}) = 1$ for every $e_n \in E'$. Thus, with probability 1, there exist infinitely many edges with length greater than $2q+Q$, which contradicts the claim that $X_\omega$ is $(q, Q)$-quasi-isometric to a \CAT space. 

\end{proof}

%% file: morse_sets_not_preserved.tex
\needspace{5\baselineskip}
\section{First-passage percolation and Morse sets}\label{sec:fpp_Morse}
Morse boundary was first introduced by \cite{contracting} for \CAT spaces and later generalized to all proper geodesic spaces by \cite{Morse} as the first QI-invariant subset of the visual boundary. The definition of Morse (See Definition~\ref{Def:Morse}) ensures quasi-isometric invariance. Let $\alpha$ be a Morse geodesic ray in $X$. We begin by showing that for a set of measure one, infinitely many geodesic segments along $\alpha$ in $X$ are no longer distance-minimizing segments in $X_\omega$.

\begin{proposition} \label{prop:morseprob1}
    Suppose $\alpha$ is a Morse geodesic ray starting at some vertex $\go\in\alpha$.
    Let $\{R_j\}_{j=1}^\infty$ be a sequence such that $ 0<R_1<R_2<\dots$ and $R_j\to \infty$ as $j\to\infty$. Let $B_{R_j}(\go)$ denote the ball of radius $R_j$ centered about $\go$ for all $j\geq 1$. Suppose $\alpha$ has the following property: There exist a fixed $n\geq 1$, and fixed $k\geq n$ such that for any ball, $B_{R_j}(\go)$, there exist two vertices $x_j,~y_j$ on $\alpha$ with the following properties:
        \begin{enumerate}
            \item $d(x_j,y_j) \geq n,$
            \item $x_j,y_j \notin B_{R_j}(\go)$,
            \item there exists a path, $\gamma_j$, of length $k$ joining $x_j$ and $y_j$,
            \item $\gamma_j$ is disjoint from $\alpha$ except at the endpoints $x_j,y_j$,
        \end{enumerate}%
    Suppose $\nu(\{0\}) = 0$, Supp$(\nu)= (0,\infty)$, and $\mathbb E \omega_e <\infty$. Then, there exists some $c>0$ that does not depend on $j$ such that
    \[\mathbb{P}\left(\big|\alpha|_{[x_j,y_j]}\big|_\omega \geq \big|\gamma_j\big|_\omega \right) \geq c. \]
    In particular,\[\mathbb{P}\left(\big|\alpha|_{[x_j,y_j]}\big|_\omega \geq \big|\gamma_j\big|_\omega ~\text{i.o.}\right) = 1.\]
\end{proposition}

\begin{proof}
    Let $j \geq 1$. Then, there exist some $n\geq 1$ and $k\geq n$ such that there exist vertices $x_j,y_j\in \alpha$ with $d(x_j,y_j) = n$ and $x_j,y_j\notin B_{R_j}(\go)$. Since $\alpha$ is a geodesic, the segment $\alpha|_{[x_j,y_j]}$ is a geodesic with length $n$. Further, there is a path $\gamma_j$ of length $k$ joining $x_j$ and $y_j$. Let $D >0$. Then
    \begin{align*}
    \mathbb{P}\left(\big|\alpha|_{[x_j,y_j]}\big|_\omega \geq \big|\gamma_j\big|_\omega \right) &\geq \mathbb{P}\left(\big|\alpha|_{[x_j,y_j]}\big|_\omega \geq \big|\gamma_j\big|_\omega \geq D \right) \\
        &= \mathbb{P}\left(\sum_{e\in \alpha|_{[x_j,y_j]}}\omega(e) \geq \sum_{e\in \gamma_j }\omega(e) \geq D\right)\\
        &\geq \mathbb{P}\left(\sum_{i=1}^n\omega(e_i) \geq \sum_{l=1}^k\omega(e_l) \geq D\right)\\
        &\geq \mathbb{P}\left(\sum_{i=1}^n\omega(e_i) \geq 3D,~2D\geq \sum_{l=1}^k\omega(e_l) \geq D\right)\\
        &= \mathbb{P}\left(\sum_{i=1}^n\omega(e_i) \geq 3D\right)\mathbb{P}\left(2D\geq \sum_{l=1}^k\omega(e_l) \geq D\right)\\
        &\geq \mathbb{P}\left(\omega(e_n) \geq 3D\right)P\left(2D/k \geq\omega(e_k) \geq D/k\right)^k\\
        &= (1-F_{\omega(e)}(3D))(F_{\omega(e)}(2D/k) - F_{\omega(e)}(D/k))^k\\
        &= c_1c_2^k = c>0
    \end{align*}
    as Supp$(\nu) = (0,\infty)$. Thus
    \begin{align*}
            \sum_{j=1}^\infty \mathbb{P}\left(\big|\alpha|_{[x_j,y_j]}\big|_\omega \geq \big|\gamma_j\big|_\omega  \right) \geq \sum_{j=1}^\infty c = \infty.
    \end{align*}
    Therefore, by the second part of the Borel--Cantelli lemma (\ref{lem:b-c}),
    \[
    \mathbb{P}\left(\big|\alpha|_{[x_j,y_j]}\big|_\omega \geq \big|\gamma_j\big|_\omega ~\text{i.o.}\right) = 1.
    \]
\end{proof}
Observe that Proposition \ref{prop:morseprob1} implies that for $a.e.~\omega$, infinitely often, the $\omega$-geodesic connecting points $x_j,y_j\in \alpha$ will take a path consisting of at least one edge not contained in $\alpha$. Further, the $\omega$-length of the geodesic segment $\alpha\big|_{[x_j,y_j]}$ will be longer than the $\omega$-length of $\gamma_j$, which will, in turn, be longer than D. Note this does not provide any information on the $\omega$-length of the $\omega$-geodesic connecting $x_j$ to $y_j$.

We are now ready to prove Theorem \ref{introthm5},

\begin{theorem} \label{Thorem C}
   Suppose $\alpha$ is a Morse quasi-geodesic ray with the assumption of the Proposition ~\ref{prop:morseprob1}. Then for $a.e.~\omega$, the image of $\alpha$ (i.e. $\omega(\alpha$)) is not a Morse set.  
\end{theorem}
 \begin{proof} By the definition of Morse \ref{Def:Morse}, it is sufficient to show the case where $\alpha$ is in fact a geodesic ray, as any Morse quasi-geodesic is at most bounded distance from a Morse geodesic. Suppose by way of contradiction that $\omega(\alpha)$ admits Morse gauge $M(q,Q)$ with $M(1,0) =D$. For each $j\geq 1$, let $\alpha^j_\omega$ be the $\omega$-geodesic joining $x_j$ and $y_j$. By assumption, $\omega(\alpha)$ is a Morse set. Therefore, for every $j\geq 1$, $\alpha_\omega^{j}$ will be contained in a $D$-neighborhood of $\omega(\alpha)$ in $X_\omega$ as each $\alpha_\omega^j$ is a (1,0)-quasi-$\omega$-geodesic joining $x_j$ to $y_j$. Where $x_j$ and $y_j$ satisfy the conditions from Proposition ~\ref{prop:morseprob1}. Therefore, by Proposition ~\ref{prop:morseprob1}, 
  \[
 \mathbb{P}\left(\big|\alpha|_{[x_j,y_j]}\big|_\omega \geq \big|\gamma_j\big|_\omega  ~\text{i.o.}\right) = 1.
 \]
 
 Thus, for a.e $\omega$, infinitely often, the $\omega$-geodesic, $\alpha_{\omega}^{j}$, contains at least one edge that is not in $\omega(\alpha)$.

 There exists a vertex $p_j \in \alpha_\omega^j$ such that $p_j \notin \omega(\alpha).$ Thus,  $d_\omega(p_j,\omega(\alpha)) \leq D.$ In particular, there is a path $\sigma_j$ connecting $p_j$ to $\omega(\alpha)$ which contains at least one edge, $e_j$ with $d_\omega(p_j,\omega(\alpha)) =\big|\sigma_j\big|_\omega \geq \big|e_j\big|_\omega$. Let $A_j = \{\big|e_j\big|_\omega \geq 3D\}$. Since Supp$(\nu) = (0,\infty)$, $\mathbb{P}(A_j) = \epsilon>0$.  Then, by the second part of the Borel--Canteli Lemma, ~\ref{lem:b-c}, $\mathbb{P}(A_j  ~\text{ i.o.}) = 1$. Thus, for $a.e.~\omega$, $d_\omega(p_j,\omega(\alpha))>3D$ for infinitely many $j\geq 1$. This contradicts $\omega(\alpha)$ admitting Morse gauge $M$.
\end{proof}
 
 The next corollary shows that, in the setting of a finitely generated group \(G\) with Cayley graph \(\Cay(G, S)\), the hypotheses of Proposition~\ref{prop:morseprob1}  are not particularly restrictive. We say that an infinite-order element \(g \in G\) is a \emph{Morse element} (with respect to \(\Cay(G, S)\)), if the orbit map $\alpha \colon \mathbb{N} \longrightarrow \Cay(G, S) \text{ such that } \alpha(n) = g^{n},$ based at the identity element $id_G$, is a \emph{Morse quasi-geodesic}.

\begin{corollary}
Let \(G\) be a finitely generated group with Cayley graph $Cay(G, S)$, and let \(g \in G\) be a Morse element with orbit $\alpha$. Suppose that there exists a path in $Cay(G, S)$ that meets $\alpha$ only at its 
endpoints. Then, for a.e. $\omega$, the image $\omega(\alpha)$ admits no Morse gauge. 
\end{corollary}
 \begin{proof}
    Let $\go = id_G$. The orbit $\alpha$ contains a path $\gamma$ joining points $x,y \in \alpha$ and meeting $\alpha$ only at its endpoints. Suppose $d(x,y)=n$ and $|\gamma|=k\ge n$. Let $\{R_j\}_{j=1}^\infty$ be an increasing sequence with $R_j\to\infty$. Since left multiplication by $g$ is an isometry
of $\mathrm{Cay}(G,S)$ and preserves $\alpha$, for each $j$ choose $t_j\in\mathbb{Z}$ large enough so that $x_j:=g^{t_j}x$ and $y_j:=g^{t_j}y$ lie outside $B_{R_j}(\go)$, and set $\gamma_j:=g^{t_j}\gamma$. So for each $j$, we have
\begin{enumerate}
  \item $d(x_j,y_j)=d(x,y)=n$,
  \item $x_j,y_j \notin B_{R_j}(\go)$,
  \item there exists a path $\gamma_j$ of length $k$ joining $x_j$ and $y_j$,
  \item $\gamma_j$ is disjoint from $\alpha$ except at the endpoints $x_j,y_j$.
\end{enumerate}
     Therefore, by Theorem \ref{Thorem C}, for $a.e.~\omega$, $\omega(\alpha)$ admits no Morse gauge.
 \end{proof}

%% file: radial_map.tex
\needspace{5\baselineskip}
\section{First passage percolation is a radial map}

In this section, we use some of the tools developed in \cite{BT17} to show that first passage percolation is $a.s.$ a radial map on a large class of graphs. For this section, we no longer consider the geometric realization of $X_\omega$. Henceforth, we only consider $X_\omega$ as a combinatorial graph equipped with random edge metric $d_\omega$.

    \begin{definition}
        [\cite{DV15}, Definition 2.4]Suppose $X$ and $Y$ are metric spaces. A function $f:X\to Y$ is \emph{large scale Lipschitz} (LSL, or ($\lambda_1,\mu_1$)-LSL) if there exists $\lambda_1,\mu_1 >0$ such that 
        $$d_Y(f(x),f(y)) \leq \lambda_1d_X(x,y) + \mu_1, \quad \forall~x,y\in X.$$
    \end{definition}

        \begin{definition}
        Suppose $X$ and $Y$ are metric spaces. A function $f: X\to Y$ is \emph{radial} (($\lambda_2,\mu_2$)-radial) with respect to $\go \in X$ if it is large scale Lipschitz and there exists $\lambda_2,\mu_2>0$ such that for every finite geodesic $\gamma:[0,M]\to X,$ with $\gamma(0) = \go$ the following condition holds:
        $$\lambda_2~d(x,y) - \mu_2\leq d(f(x),f(y)),\quad \forall x,y\in \text{Im($\gamma$)}.$$
    \end{definition}

    \begin{lemma}\label{lem:lsl}
        Let $X=(V,E)$ be a connected graph. Assume $0<b = \mathbb{E}\omega_e < \infty$. Then for a.e. $\omega$ there exists $r_0 = r_0(\omega) > 0$ such that
        \[d_\omega(x,y) \leq 2bd(x,y) + r_0\] for every $x,y \in V$.
        \begin{proof}
            Let $x,y \in X$ and let $\gamma$ be a geodesic connecting $x$ to $y$. Then $d(x,y) = \bigl|\gamma\bigr|$. Let $\Gamma$ be the set of all self-avoiding paths connecting $x$ to $y$. Then, by Proposition \ref{prop:upper_bnd}, for $a.e.~\omega$ we have \[d_\omega(x,y)=\inf_{\gamma'\in\Gamma}\bigl|\gamma'\bigr|_\omega\leq \bigl|\gamma\bigr|_\omega \leq 2b\bigl|\gamma\bigr|+r_0= 2b~d(x,y) + r_o.\]
            Thus, we have shown that for $a.e.~\omega$, $\omega$ is a $(2b,r_0)$-LSL map.
        \end{proof}
    \end{lemma}
    
\begin{lemma}\label{lem:overthought}
        Let $X=(V,E)$ be an infinite connected graph with $q<\infty$ as an upper bound on its degree. Suppose $\nu(\{0\})=0$.
        Then there exists an increasing function $\beta:(0,\infty)\to (0,q-1]$ such that $\lim_{t\to 0} \beta(t) = 0$ and such that for all sufficiently small $\epsilon>0$ and for all $x,y\in V$
        \[
        \mathbb{P}(d_\omega(x,y) \leq \epsilon~d(x,y)) \leq 2\beta(\epsilon)^{d(x,y)}.
        \]
        \begin{proof}
        Let $x,y\in X$ such that $d(x,y) = n$ and let $\Gamma$ be the set of all self-avoiding paths in $X$ connecting $x$ to $y$. Let $\gamma$ denote a geodesic connecting $x$ to $y$. Now, observe that 
            \[
                d_\omega(x,y) = \inf_{\gamma'\in \Gamma}\bigl|\gamma'\bigr|_\omega.
            \]
            For $a.e.$ $\omega$, the infimum is attained by some $\gamma_\omega \in \Gamma$ because $X_\omega$ is almost surely a geodesic space. Note that $\bigl| \gamma_\omega \bigr| \geq n$ as in the standard metric, $\gamma_\omega$ is just a self-avoiding path connecting $x$ to $y$.
            By Proposition \ref{shrinkingprop} for a path $\gamma$ of length $n$, we have $P(\bigl|\gamma\bigr|_\omega \leq \epsilon n)<\alpha(\epsilon)^n$. Since $\alpha(\epsilon)\rightarrow 0$ as $\epsilon \rightarrow 0$, there exist $\delta$ such that for every $\epsilon \leq \delta$ we have $0<q\alpha(\epsilon)<1/2.$
            Choose $\epsilon\leq \delta$. Assume $d_\omega(x,y)  \leq \epsilon n$. Then we have $\bigl|\gamma_\omega\bigr|_\omega \leq \epsilon n$ as $\bigl|\gamma_\omega\bigr|_\omega = d_\omega(x,y)$. 

            Let $A_n$ be the event $\left\{d_\omega(x,y) \leq \epsilon n =\epsilon d(x,y)\right\}$ and let $\Gamma_k$ be the set of all self-avoiding paths of length $k$ connecting $x$ to $y$. By the bound on the degree of $X$, we have \(\text{Card}(\Gamma_k) \leq q^k\). Then
            
 \begin{align*}
    A_n &\subseteq \bigcup_{k\ge n}~ \bigcup_{\gamma'\in\Gamma_k} \bigl\{\,\lvert\gamma'\rvert_\omega \le \varepsilon k\,\bigr\} \\[2pt]
    \implies\quad 
    \mathbb{P}(A_n) 
    &\le \mathbb{P}\left(\bigcup_{k\ge n}~ \bigcup_{\gamma'\in\Gamma_k} \{|\gamma'|_\omega \le \varepsilon k\}\right) \\
    &\le \sum_{k=n}^\infty ~ \sum_{\gamma'\in\Gamma_k} \mathbb{P}\left(|\gamma'|_\omega \le \varepsilon k\right) \\
    &\le \sum_{k=n}^\infty \ \sum_{i=1}^{q^k} \alpha(\varepsilon)^k \quad \text{(by Proposition \ref{shrinkingprop})} \\
    &= \sum_{k=n}^\infty q^k \,\alpha(\varepsilon)^k
    = \sum_{k=n}^\infty \bigl(q\,\alpha(\varepsilon)\bigr)^k \\
    &= \frac{\bigl(q\,\alpha(\varepsilon)\bigr)^n}{1- q\,\alpha(\varepsilon)} \\
    &\le 2\,\bigl(q\,\alpha(\varepsilon)\bigr)^n \qquad \text{as } q\,\alpha(\varepsilon) < \tfrac12 \\
    &= 2\,\beta(\varepsilon)^n.
\end{align*}

 Where we define $\beta:(0,\infty)\rightarrow (0,q]$ as $\beta(\epsilon) = q\alpha(\epsilon)$ and $\beta\big|_{[0,\delta]} \subset[0,1].$ Note that $\beta(\epsilon)$ is increasing and $\lim_{\epsilon\to0}\beta(\epsilon) = 0$, as $\alpha(\epsilon)$ is increasing and $\epsilon\rightarrow0$ as $\lim_{\epsilon\to 0} \alpha(\epsilon) = 0$.
 \end{proof}
 \end{lemma}

    \begin{lemma}\label{lem:radial}
        Let $X$ be an infinite connected graph with bounded degree, and let $\go$ be some vertex of $X$. Assume $\nu(\{0\}) = 0.$ Then for a.e. $\omega$ there exists a constant $c >0$ such that for $a.e.$ $\omega$, there exists $r_2 = r_2(\omega)$ such that for every finite geodesic $\gamma:[0,n] \to X$ with $\gamma(0) = \go$, the following condition holds:
        \begin{align*}
            c~d(x,y) - r_2 \leq d_\omega(x,y),\quad \forall x,y\in \text{Im($\gamma$)}.
        \end{align*}
        \begin{proof}
            Let $q$ be an upper bound on the degree of $X$. Then, the number of geodesics of length $n$ starting at $\go$ is at most $q^n$. For each geodesic $\gamma$ there there are $n \choose 2$ possible combinations of two vertices on $\gamma$ which must satisfy the inequality.
            
            On the other hand, by Lemma \ref{lem:overthought}, given a $\lambda= 1/q^{2n}{n\choose 2}$, we can find a $c>0$ such that, given one of the aforementioned geodesics $\gamma$, the probability that any two points $x,y$ in $\text{Im}(\gamma)$ satisfying $d_\omega(x,y) \leq c~d(x,y)$ is less than $\lambda$. So the probability of $\gamma$ having two vertices $x,y\in \text{Im}(\gamma)$ satisfying $d_\omega(x,y) \leq c~d(x,y)$ is less than ${n\choose2}\lambda$. Thus, the probability that any of the geodesics of length $n$ starting at $\go$ having vertices $x,y\in \text{Im}(\gamma)$ satisfying $d_\omega(x,y) \leq c~d(x,y)$ is less than $q^n {n\choose2} \lambda = q^n {n\choose2}\cdot1/q^{2n}{n\choose 2} = 1/q^n$. Let $A_n$ be the event that a geodesic of length $n$ starting at $\go$ satisfies $d_\omega(x,y) \leq c~d(x,y)$ for $x,y\in \{\gamma(0), \gamma(1),...,\gamma(n)\}$. Then $\sum_{n=1}^\infty \mathbb{P}(A_n) \leq \sum_{n=1}^\infty 1/q^n <\infty$. So, by the Borel--Cantelli lemma ~\ref{lem:b-c}, $\mathbb{P}(A_n~\text{i.o.}) = 0$.
        \end{proof}
    \end{lemma}
    Combining  Lemma \ref{lem:radial} and  Lemma \ref{lem:lsl} we have shown 
\begin{corollary}\label{cor:radial}
 For $a.e.~\omega$, FPP is a $(c,r_2)$-radial map with respect to $\go$.
\end{corollary}

    \begin{definition}
        Let $\gamma:[0,n]\to X$ be a geodesic path in $X$ such that $\gamma(0) = \go$ and $\gamma(n) = x_n$ for some $x_n\in V$. We define the \emph{$\omega$-velocity} at $x_n$ to be 
        \[
        v(x_n) = \frac{d_\omega(\go,x_n)}{d(\go,x_n)} = \frac{d_\omega(\go,x_n)}{n}.
        \]
    \end{definition}

    \begin{corollary} \label{cor:coarse_shape}
    Let $X$ be an infinite connected graph with bounded degree, and let $\go$ be a vertex of $X$. 
    Assume $\mathbb{E}\omega_e = b < \infty$ and $\nu(\{0\}) = 0.$ 
    Then, for a.e.\ $\omega$, the $\omega$-velocity of geodesics starting at $\go$ is asymptotically bounded by constants.
    \begin{proof}
    Let $\gamma:[0,n]\to X$ be a geodesic path in $X$ with $\gamma(0) = \go$ and $\gamma(n) = x_0$ for some $x_n\in V$.
    
    By Lemma~\ref{lem:lsl},
    \begin{align*}
    d_\omega(\go,x_n) &\leq 2b\,d(\go,x_n) + r_0,\\[4pt] 
    d_\omega(\go,x_n) &\leq 2bn+r_0,\\[4pt]
    \frac{d_\omega(\go,x_n)}{n} &\leq 2b + \frac{r_0}{n},\\[4pt]
    v(x_n) &\leq 2b + \frac{r_0}{n}.
    \end{align*}
    
    On the other hand, by Lemma~\ref{lem:radial},
    \begin{align*}
    cd(\go,x_n) - r_2 &\leq d_\omega(\go,x_n),\\[4pt]
    cn - r_2 &\leq d_\omega(\go,x_n),\\[4pt]
    c - \frac{r_2}{n} &\leq \frac{d_\omega(\go,x_n)}{n},\\[4pt]
    c - \frac{r_2}{n} &\leq v(x_n),
    \end{align*}
    
    Combining these, we obtain
    \[
    c - \frac{r_2}{n} ~\leq~ v(x_n)~\leq ~2b + \frac{r_0}{n}.
    \]
    
    Letting $n = d(\go,x_n) \to \infty$ we obtain
    \[
    c~\leq ~v(x_\infty)~\leq~ 2b.
    \]
    \end{proof}
\end{corollary}

We note that Corollary \ref{cor:coarse_shape} gives an almost sure coarse bound on the $\omega$-velocity of points at large distance from the origin provided the limit  exists.
